\documentclass[a4paper]{article}

\title{Estimates for eigenvalues of Schr\"{o}dinger operators with complex-valued potentials}
\author{
Alexandra Enblom\\
{\small Department of Mathematics}\\
{\small Linköping University}\\
{\small SE-581 83 Link\"oping, Sweden}\\
{\small \texttt{alexandra.enblom@liu.se}}\\
}

\usepackage[utf8]{inputenc}
\usepackage[T1]{fontenc}
\usepackage{amsmath}
\usepackage{amsthm}
\usepackage{amssymb} 
\usepackage{mathrsfs} 
\usepackage{enumerate}

\theoremstyle{plain}
\newtheorem{thm}{Theorem}[section]

\newtheorem{cor}[thm]{Corollary}

\theoremstyle{definition}

\newtheorem{rem}[thm]{Remark}
\newtheorem{example}[thm]{Example}

\newcommand{\reals}{\ensuremath{\mathbb{R}}}
\newcommand{\complex}{\ensuremath{\mathbb{C}}}

\newcommand{\spec}[1]{\ensuremath{\sigma\left(#1\right)}}

\DeclareMathOperator{\impart}{Im}
\DeclareMathOperator{\realpart}{Re}
\DeclareMathOperator{\sign}{sgn}

\newcommand{\ud}{\,d}

\numberwithin{equation}{section}

\date{}
\begin{document}
\maketitle

\begin{abstract}
New estimates for eigenvalues of non-self-adjoint  multi-dimensional Schr\"{o}dinger operators are obtained in terms of $L_{p}$-norms of the potentials. The results extend and improve those obtained previously. In particular,  diverse versions of an assertion conjectured by Laptev and Safronov are discussed. Schr\"{o}dinger operators with slowly decaying  potentials are also  considered.
\end{abstract}



\section{Introduction}\label{sec:introduction}
In this paper we discuss estimates  for eigenvalues of Schr\"{o}dinger operators with complex-valued potentials. Among  existing results on this problem  regarding non-self-adjoint  Schr\"{o}dinger operators we mention the works \cite{abramov-aslanyan-davies}, \cite{frank-laptev-seiringer}, \cite{frank}, \cite{safronov1}, \cite{safronov2}, and also \cite{davies} for an overview  on certain aspects of spectral analysis  of  non-self-adjoint  operators  mainly needed  for problems in quantum mechanics.  In \cite{abramov-aslanyan-davies} it was observed that  for the one-dimensional Schr\"{o}dinger operator $H = - \ud^{2}/\ud x^{2} + q$,    where   the potential $q$ is a complex-valued function belonging to $L_{1} (\reals)  \cap  L_{2} (\reals),$  every its eigenvalue $\lambda$ which does  not lie on the non-negative  semi-axis satisfies the following estimate
\begin{equation}\label{eq:eigenv}
| \lambda |^{1/2}  \leq \frac{1}{2}   \int_{- \infty}^{\infty}   | q (x) | \ud x.
\end{equation}
For the self-adjoint case the estimate \eqref{eq:eigenv} was pointed out previously by Keller in \cite{keller}.  In \cite{frank-laptev-seiringer} related estimates are found  for eigenvalues of Schr\"{o}dinger operators  on semi-axis with   complex-valued potentials.  Note  that, as is pointed out in \cite{frank-laptev-seiringer}, the obtained estimates are  in sense sharp for both  cases of Dirichlet and Neumann  boundary conditions.  In \cite{frank}, \cite{safronov1} (see also \cite{safronov2})  the problem  is considered for higher dimensions  case. In particular,  in \cite{frank}  estimates for eigenvalues of Schr\"{o}dinger operators with complex-valued potentials decaying  at infinity,  in a certain sense, are  obtained in terms of appropriate  weighted Lebesgue spaces norms  of potentials. 

 In this paper we mainly deal with the evaluation of  eigenvalues of multi-dimensio\-nal  Schr\"{o}dinger operators. The methods which we apply allow us to consider the Schr\"{o}dinger operators acting in one of the Lebesgue space   $L_{p} (\reals^{n})$ $(1 < p < \infty).$ We  consider the formal differential operator $- \Delta + q$ on $\reals^{n},$  where  $\Delta$ is the $n$-dimensional Laplacian and $q$ is a  complex-valued  measurable  function. Under some reasonable conditions, ensuring, in a suitable averaged sense, decaying at infinity of  the potential, there exists a closed  extension $H$ of  $- \Delta + q$  in the space  $L_{p} (\reals^{n}) $ such that its  essential spectrum $\sigma_{ess}(H)$ coincides with  the semi-axis $[0, \infty),$      and any other  point of the spectrum, i.e. not  belonging to  $\sigma_{ess}(H)$, is an isolated  eigenvalue of finite (algebraic)  multiplicity.  We take the operator  $H$ as the Schr\"{o}dinger operator  corresponding to $- \Delta + q$  in above  sense and we will be interested  to find estimates of  eigenvalues of $H$ which lie outside of the essential spectrum.  The problem  reduces to estimation of the resolvent of the unperturbed operator $H_{0}$, that is defined by $- \Delta$ in $L_{p} (\reals^{n})$ on its  domain the Sobolev space  $W_{p}^{2} (\reals^{n}),$   bordered by some suitable  operators of multiplication (cf. reasoning in Section 2). 

We begin with three dimensional  Schr\"{o}dinger operators (cf. Section \ref{sec:three dim}). In  this case the resolvent $R (\lambda; H_{0})$ of $H_{0}$  is an integral operator  with the kernel  $exp (- \mu | x - y |)/4 \pi | x - y |$, where $\mu = - i \lambda^{1/2}$, with for instance $\impart \lambda^{1/2} > 0$.  Due to this fact the evaluation of the bordered resolvent of the unperturbed  operator $H_{0}$ are made by applying  direct standard methods. 
 For the higher dimensional case the approach used in the proofs concerning  Schr\"{o}dinger operators on $\reals^{3}$  is not so convenient to apply.  Instead we propose other methods of obtaining bounds for eigenvalues. These methods involve somewhat {\em heat kernels}  associated to the Laplacian (cf. Section \ref{sec:schrodinger operators}). For it could be used the kernel 
$(4 \pi i t)^{- n/2}  exp (- |x - y|^{2}/4 i t), - \infty < t < \infty,$ representing the operator-group $U (t) = exp (- i t H_{0}), - \infty < t < \infty$, and then making use of the formula expressing the resolvent 
 $R (\lambda; H_{0})$  as the Laplace transform of  $U (t)$  (see \cite{hille-phillips}). In this way we obtain a  series of estimates for  perturbed eigenvalues. In particular, supposing  that $q = a b,$ where  $a \in  L_{r} (\reals^{n}),$   $b \in  L_{s} (\reals^{n})$  for $r, s$ satisfying 
$0 < r \leq \infty$, $p \leq s \leq \infty$,  $r^{- 1} -  s^{- 1} = 1 - 2 p^{- 1}$, $2^{- 1} - p^{- 1} \leq r^{- 1} \leq 1 - p^{- 1}$ and $ r^{- 1} + s^{- 1} < 2 n^{- 1},$  for any complex eigenvalue $\lambda$ of the Schr\"{o}dinger operator $H$ with $\impart \lambda \neq 0$, we have 
\begin{equation} \label{eq:imeinge}
| \impart \lambda |^{\alpha}   \leq  (4 \pi)^{\alpha - 1}   \Gamma (\alpha)  \| a \|_{r}    \| b \|_{s},
\end{equation}
in which  $\alpha : = 1 - n (r^{- 1} + s^{- 1})/2$ ($\Gamma$ denotes the gamma function).    An immediately consequence of this result (letting  $r  = s = 2 \gamma + n, \ \gamma > 0$)  is the estimate 
\begin{equation}\label{eq:imeinge1}
| \impart \lambda |^{\gamma}   \leq  (4 \pi)^{- n/2}   \Gamma  \biggl(\frac{\gamma}{\gamma + n/2} \biggl)^{\gamma + n/2} \int_{\reals^{n}}  | q (x) |^{\gamma + n/2}  \ud x 
\end{equation}
for $n \geq 3, \gamma > 0$.  The estimate \eqref{eq:imeinge1}  is a version of a conjecture due  to Laptev and Safronov \cite{laptev-safronov}.

Estimation of eigenvalues can be made representing {\em a priori} the resolvent of $H_{0}$ in terms of Fourier transform  (see Section \ref{sec:schrodinger operators} and \ref{sec:generalization}).  The method leads,  in particular, to the following  result. Let $1 < p < \infty,$ and let $q = a b$ with   $a \in  L_{r} (\reals^{n}),$   $b \in  L_{s} (\reals^{n})$  for $0 < r, s  \leq \infty$  satisfying 
$2^{- 1} - p^{- 1} \leq  r^{- 1} \leq  1 -  p^{- 1}$, $- 2^{- 1} + p^{- 1} < s^{- 1} \leq  p^{- 1},$  and $ r^{- 1} + s^{- 1} < 2 n^{- 1}.$    Then for any eigenvalue $\lambda \in \complex \setminus [0, \infty)$  of the Schr\"{o}dinger operator $H$ there holds  
\begin{equation} \label{eq:eig}
|  \lambda |^{\alpha - n/2}   \leq  C  \| a \|_{r}^{\alpha}    \| b \|_{s}^{\alpha},
\end{equation}
where $\alpha : = (r^{- 1}  + s^{- 1})^{- 1},$ and $C$ being a constant  of the potential (it is controlled; see Theorem \ref{thm:result5}).  Notice that for the particular case $n = 1, p = 2$ and  $r = s = 2$ one has  $\alpha = 2$ and $C = 1/2$, and the estimate  \eqref{eq:eig} reduces  to \eqref{eq:eigenv}. 
 From \eqref{eq:eig} it can be derived  estimates  for eigenvalues of Schr\"{o}dinger operators with  decaying potentials.   So,  for instance, taking $a (x) = (1 + | x |^{2})^{- \tau/2}$   $(\tau > 0)$,  under suitable restrictions on $r$ and $\tau$,  for an  eigenvalue $\lambda \in \complex \setminus [0, \infty)$ there holds 
\begin{equation}\label{eq:eigen int}
|  \lambda |^{r - n}   \leq  C  \int_{\reals^{n}}  | (1 + | x |^{2})^{\tau/2}  | q (x) |^{r} \ud x.
\end{equation}
In connection with \eqref{eq:eigen int} we note  the related results obtained in \cite{frank} and \cite{safronov1} (see also \cite{safronov2} and \cite{davies-nath}).

Finally, estimates obtained for  Schr\"{o}dinger operators can be successfully extended to  polyharmonic operators 
$$H_{q, m} = (- \Delta)^{m} + q,$$
in which (the  potential) $q$ is  a complex-valued measurable function,  and  $m$ is an arbitrary positive  real number.  For the eigenvalues   $\lambda \in \complex \setminus [0, \infty)$ of an operator of  this class it can be proved that 
\begin{equation}\label{eq:eig int}
|  \lambda |^{\gamma}   \leq  C  \int_{\mathbf{R}^{n}}    | q (x) |^{\gamma + n/2 m} \ud x
\end{equation}
for $\gamma > 0$ if $n \ge 2 m$ and $\gamma \ge 1 - n/2 m$ for $n < 2 m.$ The estimate given in \eqref{eq:eig int} is in fact a  result analogous to  the mentioned conjecture of Laptev  and Safronov \cite{laptev-safronov} raised for  Schr\"{o}dinger operators.

The paper consists of five sections: Introductions;  Preliminaries. Setting of the problem; Three dimensional  Schr\"{o}dinger operators;   Schr\"{o}dinger operators  on $\reals^{n}$; Polyharmonic operators.

\section{Preliminaries. Setting of the problem}\label{sec:preliminaries}

Consider, in the space  $L_{p} (\reals^{n})$ $(1 < p < \infty),$  the  Schr\"{o}dinger operator
\begin{equation}\label{eq:schrod}
- \Delta + q (x)
\end{equation}
with a potential $q$ being in general  a complex-valued  measurable function  on $\reals^{n}$.   We assume that the  potential $q$ admits a  factorization $q  = a  b$ with $a, b$ belonging to some Lebesgue type spaces (appropriate spaces will be  indicated in relevant  places). We denote by $H_{0}$ the operator defined by $- \Delta$ in $L_{p} (\reals^{n})$  on its domain the Sobolev space $W_{p}^{2} (\reals^{n})$,   and  let $A, B$ denote, respectively, the  operators of multiplication by $a, b$ defined in  $L_{p} (\reals^{n}) $  with the  maximal domains.  Thus, the  differential expression \eqref{eq:schrod} defines  in the space $L_{p} (\reals^{n}) $  an operator expressed as the perturbation of $H_{0}$ by $A B$.  
  In order to determine the operator,  being a closed extension of $H_{0} + A B,$  suitable for our purposes, we  need to require certain assumptions on the potential.  For we let $a$ and $b$ be functions of Stummel classes \cite{stummel} (see also \cite{jorgens-weidmann} and \cite{schechter1}),   namely 
	
\begin{equation}\label{eq:Ma}
M_{\nu, p^{'}} (a) < \infty, \quad 0 < \nu < p^{'},
\end{equation}
	
\begin{equation}\label{eq:Mb}
M_{\mu, p} (b) < \infty, \quad 0 < \mu < p,
\end{equation}	

($p^{'}$ is the conjugate exponent to $p : p^{- 1} + p^{' - 1}  = 1$), where it is denoted
$$M_{\nu, p} (u) = \sup_{x} \int_{| x - y | < 1}  | u (y) |^{p} \ | x - y |^{\nu - n} \ud y$$
for functions $u \in L_{p, loc} (\reals^{n}).$    If also the potential $q$ decays at infinity, for instance, like 

\begin{equation}\label{eq:qdecays}
 \int_{| x - y | < 1}  | q (y) |  \ud y \to 0 \quad as \quad | x | \to \infty,
\end{equation}
then the operator $H_{0} + A B  \quad (= - \Delta + q)$ admits a closed extension $H$ having  the same essential spectrum as  unperturbed operator $H_{0}$, i.e.,
$$\sigma_{ess}(H)  = \sigma_{ess} (H_{0}) \quad (=  \spec {H_{0}} = [0, \infty)).$$

 Note  that the conditions \eqref{eq:Ma} and \eqref{eq:Mb} are used to derived  boundedness  and also, together with \eqref{eq:qdecays}, compactness  domination  properties of the perturbation  (reasoning are due to Rejto \cite{rejto} and Schechter \cite{schechter2}, cf. also \cite{schechter1}; Theorem 5.1, p.116). To be more precise, due to conditions \eqref{eq:Ma} and \eqref{eq:Mb}, the bordered  resolvent $B R (z; H_{0})  A$ $(R (z; H_{0}) : = (H_{0} - z I)^{- 1}$  denotes  the resolvent of $H_{0}$)  for some (or,  equivalently, any) regular point $z$ of $H_{0}$  represents a densely defined operator  having a (unique) bounded extension, further on we denote it by $Q (z)$.    If, in addition, \eqref{eq:qdecays}, $Q (z)$ is a compact operator and, moreover, it  is small with respect to the  operator norm for sufficiently  large $| z |$.

From now on we let $H$ denote the Schr\"{o}dinger operator realized in this way in  $L_{p} (\reals^{n})$  by the differential  expression $- \Delta + q (x).$   Notice  that constructions related to that  mentioned above are widely  known in the perturbation theory. In Hilbert case space $p = 2,$ $H$,  where the potential $q$ is a real function,  represents a self-adjoint operator presenting mainly interest for spectral and scattering problems.

It turns out that there is a constraint relation between the discrete part of the spectrum of $H$ and that of  $Q (z)$ (recall $Q (z)$ is the bounded  extension of the bordered resolvent  $B R (z; H_{0}) A$),    namely, a regular  point $\lambda$ of $H_{0}$ is an eigenvalue for the extension $H$, the Schr\"{o}dinger operator, if and only if $- 1$ is an eigenvalue of  $Q (\lambda).$   This fact, which will  play a fundamental role in our  arguments, can be deduced  essentially, by corresponding  accommodation to the situation of Banach space case, using similar arguments as in the proof of Lemma 1 \cite{konno-kuroda}.   
 Consequently, for an eigenvalue $\lambda$ of the Schr\"{o}dinger operator $H$,  $\lambda$ being  a regular point of the unperturbed  operator $H_{0}$,  the operator norm of $Q (\lambda)$ must be no less  than $1$,  i.e.,
$\|  Q (\lambda) \| \geq 1.$ Namely from this operator norm evaluation we will  derive estimates for eigenvalues  of the  Schr\"{o}dinger operator $H$.
  
 Throughout the paper  there will always assumed (tacitly)  that  the conditions \eqref{eq:Ma}, \eqref{eq:Mb} and \eqref{eq:qdecays}  are satisfied.

\section{Three dimensional  Schr\"{o}dinger operators}\label{sec:three dim}

We first consider the case $n = 3$. In this case the fundamental solution of the  operator $H_{0}   - \lambda$ $(= - \Delta - \lambda)$ in $\reals^{3}$,  i.e.,  the solution  $\Phi \in S^{'}  (\reals^{3})$  of the  equation
$$- (\Delta + \lambda) \Phi (x) = \delta (x), \quad x \in \reals^{3},$$
is expressed explicitly by
$$\Phi (x) =  \frac{1}{4 \pi | x |}  e^{- \mu |x|},   \quad x \in \reals^{3},$$
where $\mu = -  i \lambda^{1/2}$   and   $\lambda^{1/2}$    is chosen so that $\impart   \lambda^{1/2}   > 0.$    Consequently,  the resolvent $R (\lambda; H_{0})  : = (H_{0} - \lambda)^{- 1}$
of $H_{0}$   is an integral operator with the  kernel 
$$ \frac{1}{4 \pi | x - y |}  e^{- \mu |x - y|},$$
that will make useful in evaluation of the bordered resolvent of $H_{0}$.

There  holds the following result.

\begin{thm}\label{thm:result1}
Let $1 < p < \infty,$  and let $q = a b$ with $a \in   L_{r} (\reals^{3} )$  and $b \in L_{s} (\reals^{3})$ for $0 < r \leq \infty,$ $p \leq  s \leq \infty$ such  that  $r^{-1} +  s^{- 1}  < 2/3.$
Then, for any eigenvalue $\lambda \in \complex \setminus [0, \infty)$   of  the Schr\"{o}dinger operator $H$, considered acting in the space $ L_{p} (\reals^{3}),$    there holds 
\begin{equation}\label{eq:res1}
| \lambda |^{(3 - \alpha)/2} \leq C  (r, s, \theta)  \| a \|_{r}^{\alpha} \ \| b \|_{s}^{\alpha}, 
\end{equation}
where $C (r, s, \theta) = (4 \pi)^{1 - \alpha}  \Gamma (3 - \alpha) (\alpha \sin (\theta/2))^{\alpha - 3}$
 ($\Gamma$ denotes the gamma function),  $\alpha : =   (1 -  r^{- 1} - s^{- 1})^{- 1}$  and $\theta : = arg   \lambda$  $(\in (0, 2 \pi)).$
 \end{thm}
\begin{proof}
We have to show the  boundedness of  the operator $Q (\lambda)  = B R (\lambda; H_{0})  A$   and  evaluate its norm. Note that $Q (\lambda)$ is   an integral operator with kernel 
$$ \frac{1}{4 \pi | x - y |}  e^{- \mu |x - y|}   a (y) b (x).$$
In order to evaluate this integral operator  we first  observe that,  under supposed  conditions, the operator of multiplication   $A$  is bounded viewed as an operator from  $ L_{p} (\reals^{3} )$  to $ L_{\beta} (\reals^{3} )$ with some $\beta \geq 1.$   In fact, since $ a \in  L_{r} (\reals^{3} ),$  for any $u \in  L_{p} (\reals^{3} ),$  by H\"{o}lder's inequality,  we have
\begin{equation}\label{eq:anorm}
\| a u \|_{\beta}   \leq   \| a \|_{ r}   \| u \|_{p}, \quad    \beta^{- 1} =  r^{- 1} + p^{- 1}.
\end{equation}
Similarly, one can choose a $\gamma$ with  $p \leq \gamma \leq \infty,$ for which
\begin{equation}\label{eq:bnorm}
 \| b v \|_{p}    \leq   \| b \|_{ s}    \| v \|_{\gamma}, \quad  \gamma^{- 1}  +  s^{- 1}  = p^{- 1},
\end{equation}
for $v \in  L_{\gamma} (\reals^{3})$,  that means  that $B$ represents  a bounded operator from $ L_{\gamma} (\reals^{3} )$ to $ L_{p} (\reals^{3} ).$

Now, we observe that  the function 
$$g (x; \lambda)  =    \frac{1}{4 \pi | x |}  e^{- \mu |x|}, \quad  x \in  \reals^{3},$$
belongs to the class $ L_{\alpha} (\reals^{3})$ and,  moreover, 

\begin{equation}\label{eq:gnorm}
\| g  (\cdot; \lambda) \|_{\alpha}  = (4 \pi)^{(1 - \alpha)/\alpha}  (\alpha \ Re \ \mu)^{(\alpha -3)/\alpha} \  ( \Gamma (3 - \alpha) )^{1/\alpha}
\end{equation}
 In fact, by using the polar coordinates  
$\rho = | x |$,   $\omega = x / | x | \in S_{2}$
 ($S_{2}$  denotes the unit sphere in $\reals^{3}$),  one has 

$$\| g  (\cdot; \lambda) \|_{\alpha}^{\alpha}  =  \int_{\reals^{3}}   \frac{1}{(4 \pi | x |)^{\alpha}}  e^{- \alpha (\realpart \mu) |x|}  \ud x $$ 
$$=   (4 \pi)^{- \alpha}  \int_{0}^{\infty}    \int_{S_{2}}  \rho^{-\alpha +2}  e^{- \alpha (\realpart \mu) \rho}  \ud \rho \ud \omega $$

$$=    (4 \pi)^{- \alpha}   mes (S_{2})    \int_{0}^{\infty}  \rho^{- \alpha + 2}   e^{- \alpha (\realpart  \mu) \rho}  \ud \rho. $$ 
  
Since $r^{- 1} + s^{- 1} < 2/3,$   that implies  $3 - \alpha > 0$, and  since   $\realpart \mu > 0$   (so was chosen $\mu$),  the formula  (see \cite{gradshteyn-ryzhik}; 3.381.4., p.331)

\begin{equation}\label{eq:real}
\int_{0}^{\infty}  x^{\nu - 1} e^{- \mu x}   \ud x =  \mu^{-\nu}   \Gamma (\nu), \quad \realpart  \nu > 0, \quad \realpart  \mu > 0
\end{equation}
can be used, and we find
$$\int_{0}^{\infty}  \rho^{- \alpha + 2}   e^{- \alpha (\realpart  \mu) \rho}  \ud \rho  =   (\alpha \realpart \mu)^{\alpha -3}   \Gamma (3 - \alpha).$$
Since $mes (S_{2}) = 4 \pi,$  we finally obtain 
$$\| g  (\cdot; \lambda) \|_{\alpha}^{\alpha}  =   (4 \pi)^{1 - \alpha}  (\alpha \realpart  \mu)^{\alpha - 3}  \Gamma (3 - \alpha),$$

i.e.,   \eqref{eq:gnorm}.

By Young's Inequality (see for instance, \cite{bergh-lofstrom}, Theorem 1.2.2, or also \cite{folland}; Proposition 8.9a)   the operator  $R (\lambda; H_{0}),$   representing an integral operator of convolution type (with  the  kernel $g (x - y; \lambda)$),   is bounded as an operator from $ L_{\beta} (\reals^{3} )$  into $ L_{\gamma} (\reals^{3} )$ provided that
\begin{equation}\label{eq:alphabeta}
\gamma^{- 1}  + 1 =  \alpha^{- 1}  + \beta^{- 1}.
\end{equation}

Moreover,

$$\| R (\lambda; H) v \|_{\gamma}  \leq \|g \|_{\alpha}    \| v \|_{\beta}, \quad v \in   L_{\beta} (\reals^{3}).$$

 Note that \eqref{eq:alphabeta} indeed  follows immediately from the relations between  $p$, $q$, $r$  and $s$    given by \eqref{eq:anorm}  and \eqref{eq:bnorm}:
$$1 - \beta^{- 1}  + \gamma^{- 1}  = 1 -  r^{- 1}  - p^{- 1}  + p^{- 1}  -  s^{- 1} = 1 -  r^{- 1}  -  s^{- 1} = \alpha^{- 1}.$$  

The evaluations  \eqref{eq:anorm}, \eqref{eq:bnorm} and \eqref{eq:gnorm} made above imply that 

$$\| Q  (\lambda) u \|_{p}  =  \| B R (\lambda; H_{0})   A u \|_{p}  \leq \| a \|_{r} \  \| b \|_{s} \ \| g \|_{\alpha}  \| u \|_{p}$$
for each $u \in L_{p} (\reals^{3}).$  Therefore, in view of \eqref{eq:gnorm}, we have 

\begin{equation}\label{eq:Qlambda}
	\|  Q (\lambda)  \|  \leq  (4 \pi)^{(1 - \alpha)/\alpha}  (\alpha \realpart \mu)^{(\alpha - 3)/\alpha}   (\Gamma (3 - \alpha))^{1/\alpha}    \| a \|_{r}     \| b \|_{s}.
\end{equation}
	
The desired estimate \eqref{eq:res1} for the eigenvalue $\lambda$ follows   from the fact that the value on the left-hand  side of \eqref{eq:Qlambda} must be at least equal to $1$ (note that $\realpart \mu = \realpart (- i \lambda^{1/2})  =  | \lambda |^{1/2} \sin (\theta/2)$   by letting   $\lambda =  | \lambda | e^{i \theta}, \ 0 < \theta < 2 \pi$).
\end{proof}

From the estimate \eqref{eq:res1} it can be derived many particular estimates useful in applications.   We begin with the situation  when $a, b \in  L_{r} (\reals^{3})$  with $r > 3$ if $1 < p \leq 3$ and $p \leq r \leq \infty$ if $p > 3.$   In \eqref{eq:res1} we can take $s = r,$  then $r^{-1 } + s^{-1} = 2 r^{-1}  (< 2/3)$   and   $\alpha = r /(r - 2).$   In view of Theorem \ref{thm:result1}, we have the following result.

\begin{cor}\label{cor:q}
Suppose    $q = a b,$ where   $a, b \in L_{r} (\reals^{3})$  with  $r > 3$ if $1 < p \leq 3$ and $p \leq r \leq \infty$ if $p > 3$. 
  Then for  any eigenvalue $\lambda \in \complex \setminus [0, \infty)$  of the Schr\"{o}dinger operator $H$, considered in $ L_{p} (\reals^{3}),$  there holds 
\begin{equation}\label{eq:cor1}
| \lambda |^{r - 3}   \leq C (r, \theta) \ \| a \|_{r}^{r} \ \| b \|_{r}^{r}, 
\end{equation}
where  
$$C (r, \theta) = (4 \pi)^{- 2} \Gamma (2 (r - 3)/(r - 2))^{r - 2} \   ((r - 2)/r \sin (\theta/2))^{2 (r - 3)},$$
in which   $\theta = arg \lambda \ (0 < \theta < 2 \pi)$.
\end{cor}

The following estimate was conjectured,  but for the case of Hilbert space $ L_{2} (\reals^{3}),$ by Laptev and Safronov \cite{laptev-safronov}.

\begin{cor}\label{cor:q1}
Let $\gamma > 0$ for $1 < p \leq 3$  and  $2 \gamma \geq p - 3$ for $p > 3$, supposing that 
$$q \in  L_{\gamma + 3/2} (\reals^{3}).$$
Then every eigenvalue 
$\lambda \in \complex \setminus [0, \infty)$  of the Schr\"{o}dinger operator $H$, considered acting in $ L_{p} (\reals^{3}),$  satisfies
\begin{equation}\label{eq:corq1}
| \lambda |^{\gamma} \leq  C (\gamma, \theta)   \int_{\reals^{3}}  | q (x) |^{\gamma + 3/2} \ud x,
\end{equation} 
where
$$C (\gamma, \theta) = (1/4 \pi \sin^{2\gamma} (\theta/2))  ((2 \gamma + 1)/(2 \gamma + 3))^{2\gamma}  \Gamma (4 \gamma/(2 \gamma + 1))^{(2\gamma + 1)/2}.$$
\end{cor}
\begin{proof}
It suffices to let $r = 2 \gamma + 3$ in \eqref{eq:cor1} and take
$$ a (x) = | q (x) |^{1/2}, \quad b (x) = (\sign  q (x))  | q (x) |^{1/2},$$
where $\sign  \ q (x) = q (x)/| q (x) |$ if $q (x) \neq 0$ and $\sign \ q (x)  = 0$  if $q (x)  = 0.$
\end{proof}  

Frank \cite{frank} also obtains a result concerning  already mentioned conjecture in case of the  Hilbert space $ L_{2} (\reals^{3})$  and with restriction $0 < r \leq 3/2.$ The proofs in \cite{frank} are based on a uniform Sobolev  inequality due to Kenig, Ruiz and Sogge \cite{kenig-ruiz-sogge}.

Another type of estimates can be obtained  directly from \eqref{eq:cor1} by involving decaying    potentials. So, for instance, if we take   $a (x) = (1 + | x |^{2} )^{- \tau/2}$  with $\tau r > 3$,  $r$   satisfying  restrictions attributed as in Corollary \ref{cor:q}, then, by using the formula (see \cite{gradshteyn-ryzhik}; 3.251.2.),
$$\int_{0}^{\infty} x^{\mu - 1}   (1 + x^{2})^{\nu - 1} \ud x =  \frac{1}{2} B (\mu/2,  1 - \nu -\mu/2), \quad \realpart \mu > 0, \quad \realpart (\nu + \mu/2)   < 1,$$
where $B (x, y)$  denotes the beta-function
$$B (x, y) =  \int_{0}^{1}  t^{x} (1  - t)^{y - 1} \ud t, \quad \realpart x > 0, \quad \realpart y > 0,$$
we can  calculate 
$$\| a \|_{r}^{r} = \int_{\reals^{3}}   ( 1  + | x |^{2} )^{- \tau r/2}  \ud x =   \int_{0}^{\infty}  \int_{S^{2}}   \rho^{2}  (1 + \rho^{2})^{- \tau r/2}   \ud \rho \ud \omega  $$
$$= 4 \pi  \int_{0}^{\infty}  \rho^{2} (1 + \rho^{2})^{- \tau r/2}    \ud \rho = 2 \pi B (3/2, \tau r/2 - 3/2),$$
and, further, taking $b (x) = (1 + | x |^{2})^{\tau/2}  q (x),$    we  obtain the following result.
\begin{cor}\label{cor:q2}
Suppose 
$$ (1 + | x |^{2})^{\tau/2}  q  \in   L_{r} (\reals^{3}),$$
where $\tau r > 3$, and $r > 3$ if $ 1 < p \leq 3$ and $p \leq r \leq \infty $   if $p > 3$.   Then every  eigenvalue $\lambda \in \complex \setminus [0, \infty)$  of the Schr\"{o}dinger operator $H$, considered acting in $ L_{p} (\reals^{3}),$  satisfies
\begin{equation}\label{eq:corq2}
| \lambda |^{r - 3} \leq  C_{1} (r, \theta)   \int_{\reals^{3}}  |  (1 + | x |^{2})^{\tau/2}    q (x) |^{r} \ud x,
\end{equation} 
where $C_{1} (r, \theta) = 2 \pi B (3/2, \tau r/2 - 3/2)  C (r, \theta)$,  $C (r, \theta)$ being determined as in \eqref{eq:cor1}.
\end{cor}

It stands to  reason that estimates of type \eqref{eq:corq2} can be given choosing other  (weight) functions, used frequently for diverse proposes, as, for instance,  $e^{\tau | x |}$,  $| x |^{\sigma}  e^{\tau | x |}$,    $e^{\tau | x |^{2}},$  etc.. We cite \cite{frank} (see also  \cite{safronov1} and \cite{safronov2} for some related results involving  weight-functions as in Corollary \ref{cor:q2}).

\begin{rem}\label{rem:beckner}
The estimate \eqref{eq:res1}  can be improved up to  a factor $(A_{\alpha}  A_{\beta} A_{\gamma^{'}})^{3}$  if in proving of Theorem \ref{thm:result1} it would be used the sharp   form of Young's convolution inequality due to Beckner \cite{beckner}, where  $A_{\alpha},  A_{\beta}$ and  $A_{\gamma^{'}}$ are defined in  accordance with the notation $A_{p} = (p^{1/p}/p^{' 1/p^{'}})^{1/2}.$  If  it turns out  that $A_{\alpha}  A_{\beta} A_{\gamma^{'}} < 1$ as, for instance,  in case $1 < \alpha, \beta, \gamma^{'} < 2,$   one has indeed an  improvement of \eqref{eq:res1}. So it happens in \eqref{eq:res1}  for the particular case $r = s = 4.$   It can be  supposed 
 $q \in L_{2} (\reals^{3})$  and, as is easily checked  for the possible eigenvalues $\lambda$,   there holds
 \begin{equation}\label{eq:rembeck}
| \lambda | \leq \frac{1}{64 \pi^{2} \sin^{2} (\theta/2)}  \| q \|_{2}^{4}.
\end{equation}
 However, as is seen, $\alpha = 2$, $\beta = \gamma^{'} = 4/3,$ hence  the constant in  \eqref{eq:rembeck}  can be improved by the  factor   
 $ 2^{3} \cdot 3^{- 3/4}.$
\end{rem}

\section{ Schr\"{o}dinger operators in $\reals^{n}$}\label{sec:schrodinger operators}

\textbf{1.} If $n > 3$, the method used in  the proof of Theorem \ref{thm:result1}  is certainly applicable, actually with major difficulties. For the general  case the fundamental solution  $\Phi (x)$  of the Laplacian $- \Delta,$  and therefore the  kernel of the resolvent $R (\lambda; H_{0}),$   is expressed by   Bessel's functions  (see, for instance, \cite{berezin-shubin}).    Of course,  the asymptotic formula

$$\Phi (x)  = c | x |^{-(n - 1)/2} e^{- \mu | x |} \ (1 + o (1)), \quad | x | \to \infty,$$

with  $c > 0$ and $\realpart \mu > 0,$  is  useful  in that work,  however we have not use this fact.
 Nevertheless, an estimate related to \eqref{eq:res1}  can be  obtained for  the general case $n \geq 3$  by using  the following  integral  representation of   the free Green  function 

\begin{equation}\label{eq:green}
g (x - y;  \lambda) = (4 \pi)^{- n/2}   \int_{0}^{\infty}
 e^{\lambda t}   e^{- | x - y |^{2}/4 t}  t^{- n/2}  \ud t, \quad \realpart \lambda < 0.
\end{equation}

In other words we use the fact that the resolvent  $R (\lambda; H_{0})$  can be represented  as a convolution operator  with the kernel \eqref{eq:green},   namely

$$(R (\lambda; H_{0}) u) (x)   =  \int_{\reals^{n}}   g   (x - y;  \lambda)  u (y)  \ud y, \quad  \realpart \lambda < 0.$$

As in the previous subsection we suppose that the potential $q$ is factorized as $q = a b,$ where     $a \in  L_{r} (\reals^{n})$  and $b \in  L_{s} (\reals^{n})$   with $0 < r \leq \infty$   and $p \leq s < \infty,$  and let $A, B$ denote the operators of multiplication  by $a, b$, respectively. By similar arguments  to those used in the proof of Theorem \ref{thm:result1},  one can obtain corresponding evaluations for  $\reals^{n}$ exactly as \eqref{eq:anorm} and \eqref{eq:bnorm}. Accordingly, $A$ can be viewed as a bounded operator  from $L_{p} (\reals^{n})$ to $L_{\beta} (\reals^{n})$  and, respectively, $B$ from $L_{\gamma} (\reals^{n})$ to $L_{p} (\reals^{n})$, where
\begin{equation}\label{eq:cond}
\beta^{- 1}  =  r^{- 1} + p^{- 1}, \quad p^{- 1} = s^{- 1} + \gamma^{- 1}.
\end{equation}

Now, we take an $\alpha \geq 1$  such that 
\begin{equation}\label{eq:cond1}	
\alpha^{- 1}  + \beta^{- 1} = \gamma^{- 1} + 1  
\end{equation}
and find conditions under which the kernel  function $g (x; \lambda)$    belongs to the space  $ L_{\alpha}   (\reals^{n}).$    By Minkowski's inequality we have 

$$\| g (\cdot ; \lambda) \|_{\alpha}  =   \biggl(  \int_{\reals^{n}}  \biggl|  (4 \pi)^{- n/2} \int_{0}^{\infty}   e^{\lambda t}   e^{- | x |^{2}/4 t} \  t^{- n/2}  \ud t  \biggl|^{\alpha}  \ud x    \biggl)^{1/\alpha}$$

$$\leq   (4 \pi)^{- n/2} \int_{0}^{\infty}   \biggl(  \int_{\reals^{n}}  \biggl|  e^{\lambda t}   e^{- | x |^{2}/4 t} \  t^{- n/2}    \biggl|^{\alpha}  \ud x    \biggl)^{1/\alpha}  \ud t $$

$$=  (4 \pi)^{- n/2} \int_{0}^{\infty}   \biggl(  \int_{\reals^{n}}   e^{- \alpha | x |^{2}/4 t}   \ud x    \biggl)^{1/\alpha}  \  e^{(\realpart \lambda) t}     t^{- n/2}  \ud t,$$

and since
$$\int_{\reals^{n}}   e^{- \alpha | x |^{2}/4 t}   \ud x   =   (4 \pi t/\alpha)^{n/2},$$
it follows
 $$\| g (\cdot; \lambda) \|_{\alpha}   =    (4 \pi)^{- n/2} \int_{0}^{\infty}   (4 \pi t/\alpha)^{n/2 \alpha}    t^{- n/2}  e^{(\realpart \lambda) t} \ud t $$
$$=   (4 \pi)^{- n/2  \alpha^{'}}  \alpha^{- n/2 \alpha}  \int_{0}^{\infty}    t^{- n/2 \alpha^{'}}  e^{(\realpart \lambda) t} \ud t.$$

If $\alpha$ is chosen so that 
\begin{equation}\label{eq:alpha}
- \frac{n}{2 \alpha^{'}}  + 1  > 0, \quad  i.e. \quad  \alpha <  \frac{n}{n - 2},
\end{equation}
it can be applied the formula \eqref{eq:real}, and we obtain
 
$$\| g (\cdot; \lambda) \|_{\alpha}   \leq    (4 \pi)^{- n/2  \alpha^{'}}   \alpha^{- n/2  \alpha}  | \realpart \lambda |^{- 1 + n/2  \alpha^{'}}  \Gamma (1 - n/2  \alpha^{'}).$$

By Young's inequality we have
\begin{equation}\label{eq:young}
\| R  (\lambda; H_{0}) v \|_{\gamma}   \leq   \| g (\cdot; \lambda) \|_{\alpha} \ \| v \|_{\beta}, \quad  v \in  L_{p} (\reals^{n})
\end{equation}
provided that \eqref{eq:cond1}.

Thus, under supposed conditions, we obtain the   following  estimation
$$\| Q (\lambda)   \|  \leq    (4 \pi)^{- n/2  \alpha^{'}}   \alpha^{- n/2  \alpha}  | \realpart \lambda |^{- 1 + n/2  \alpha^{'}}  \Gamma (1 - n/2  \alpha^{'})   \| a \|_{r}  \  \| b \|_{s},$$
and, therefore, for each $\lambda \in  \complex$ with $\realpart \lambda < 0$  such  that  
$\| Q (\lambda) \|  \geq 1,$
in particular,  for an eigenvalue   of the Schr\"{o}dinger operator  $H,$  there holds the estimation
$$| \realpart \lambda |^{1 - n/2  \alpha^{'}}   \leq  (4 \pi)^{- n/2  \alpha^{'}}   \alpha^{- n/2  \alpha}   \Gamma (1 - n/2  \alpha^{'})   \| a \|_{r}  \  \| b \|_{s}.$$ 

Eliminating $\beta$ and $\gamma$  with  taking  into account (4.2) and (4.3), we see that  $\alpha  = (1 -  r^{- 1} -  s^{- 1})^{- 1},$
and, due to  (4.4),  with the   restriction
$r^{- 1}  + s^{- 1}  < 4 n^{- 1}. $

We have proved the following  result.
\begin{thm}\label{thm:result2}
Let  $n \geq 3,  1 < p < \infty$  and let     $q  = a b,$ where 
$a \in  L_{r}   (\reals^{n}),  b \in   L_{s}   (\reals^{n})$  with  $0 < r \leq \infty, p \leq s \leq \infty$  and  $r^{- 1}  + s^{- 1}  < 2 n^{- 1}.$   Then,   for any  eigenvalue $\lambda$   with    $\realpart \lambda < 0$  of the  Schr\"{o}dinger operator $H$, considered acting in the space $ L_{p}   (\reals^{n})$, there   holds 
\begin{equation}\label{eq:res2}
| \realpart \lambda |^{1 - n/2 \alpha^{'}}   \leq  C (n, r, s) \| a \|_{r}  \| b \|_{s},
\end{equation}
where $C (n, r, s) = (4 \pi)^{- n/2 \alpha^{'}} \alpha^{- n/2 \alpha} \Gamma (1 - n/2 \alpha^{'}), \  \alpha = (1 -  r^{- 1}  - s^{- 1})^{- 1}.$
\end{thm}

The following consequences of Theorem \ref{thm:result2} are natural extensions of the corresponding results given by Corollaries \ref{cor:q}, \ref{cor:q1} and \ref{cor:q2}.
\begin{cor}\label{cor:res2}
Suppose $q = a b,$ where $a \in  L_{r}   (\reals^{n})$ with $r > n$ if $1 < p \leq n$  and $p \leq r \leq \infty$ if $p > n$.  Then every eigenvalue $\lambda$ with $\realpart \lambda < 0$ of the Schr\"{o}dinger operator $H$, considered acting in  $ L_{p}   (\reals^{n})$, satisfies 
\begin{equation}\label{eq:corres21}
| \realpart \lambda |^{r - n}   \leq  C (n, r) \| a \|_{r}^{r}  \| b \|_{r}^{r},
\end{equation}
where $C (n, r) = (4 \pi)^{- n}  (1 - 2 r^{- 1} )^{n(r - 2)/2}   \Gamma (1 - n r^{- 1})^{r}.$
\end{cor}

Corollary \ref{cor:res2} in turn implies the conjuncture for the case $\reals^{n}$ raised by  Laptev and Safronov in \cite{laptev-safronov}. For it suffices to  let $r = 2 \gamma + n$ with suitable $\gamma$.
\begin{cor}\label{cor:res22}
Let $\gamma  > 0$ if $1 < p \leq n$ and $2 \gamma \geq  p - n$ if $p > n.$ Suppose
$$q   \in  L_{\gamma + n/2}   (\reals^{n}).$$
Then every eigenvalue $\lambda$ with $\realpart \lambda < 0$  of the Schr\"{o}dinger operator $H$, considered acting in  $ L_{p}   (\reals^{n})$, satisfies 
\begin{equation}\label{eq:corres22}
| \lambda |^{\gamma}   \leq  C (n, \gamma, \theta)  \int_{\reals^{n}}  | q (x) |^{\gamma + n/2}  \ud x, 
\end{equation} 
where
$$ C (n, \gamma, \theta)  =   \frac{1}{(4 \pi)^{n/2} \cos^{2 \gamma} \theta}   \biggl(  \frac{2 \gamma + n - 2}{2 \gamma + n} \biggl)^{n (2 \gamma + n - 2)/4}    \Gamma  \biggl(  \frac{2 \gamma}{2 \gamma + n} \biggl)^{\gamma + n/2}, $$
$\pi - \theta = arg \lambda$  \quad    $(- \pi/2 < \theta < \pi/2).$
\end{cor}

Next, we let $a (x) = (1 + | x |^{2})^{- \tau/2}$  and $b (x) =   (1 + | x |^{2})^{ \tau/2} q (x)$ for some $\tau > 0.$  It is  seen that, when $\tau r > n$, one has  $a \in  L_{r}   (\reals^{n})$  and, moreover,
$$ \| a \|_{r}^{r}   =   \pi^{n/2}   \Gamma \biggl( \frac{\tau r - n}{2}\biggl)\biggl/ \Gamma \biggl(\frac{\tau r}{2}\biggl).$$
For we apply similar arguments as in the proof of Corollary \ref{cor:q2} and use the relation   between beta and gamma functions (cf. \cite{gradshteyn-ryzhik}; 8.384.1.).

In view of Corollary \ref{cor:res2} the following result hold true.
\begin{cor}\label{cor:corres3}
Suppose 
$$(1 + | x |^{2})^{\tau/2} q  \in  L_{r}   (\reals^{n}),$$
  where  $\tau r > 0$,  and  $r > n$ if $1 < p \leq n$  and $p \leq r \leq \infty$ if $p > n$. 
 Then every eigenvalue $\lambda$ with $\realpart \lambda < 0$ of the Schr\"{o}dinger operator $H$, considered acting in  $ L_{p}   (\reals^{n})$, satisfies 
\begin{equation}\label{eq:corres2}
| \realpart \lambda |^{r - n}   \leq  C_{1} (n, r)   \int_{\reals^{n}}  | (1 + | x |^{2} )^{\tau/2}  q (x) |^{r}  \ud x,
\end{equation}
where $$C_{1} (n, r) =  \pi^{n/2}   C (n, r)   \Gamma \biggl(\frac{\tau r - n}{2}\biggl)\biggl/ \Gamma \biggl(\frac{\tau r}{2}\biggl),$$
$C (n, r)$ being determined as in \eqref{eq:corres21}.
\end{cor}

\begin{rem}\label{rem:res2}
By applying  the sharp form of  Young's inequality \cite{beckner} the estimation \eqref{eq:res2}  refines  by the factor $(A_{\alpha}  A_{\beta}  A_{\gamma^{'}})^{n},$   where  the constant   $A_{\alpha}$, $A_{\beta}$ and   $A_{\gamma^{'}}$  are defined like  in the Remark \ref{rem:beckner}. 
\end{rem}

The following example  is used for the purpose of illustration the above statement, although Schr\"{o}dinger operators with potentials belonging to the  space $ L_{n}   (\reals^{n})$,  as is assumed, occur  in certain situations important for applications.
\begin{example}\label{ex:1}
Let $n \geq 3, p = 2,$ and suppose $ q  \in  L_{n}   (\reals^{n}).$  Put $r = s = 2 n,$ and let $a (x ) = | q (x) |^{1/2},  b (x) = q (x)/| q (x) |^{1/2}.$   Then, by (4.7),  for  eigenvalues $\lambda$ with $\realpart \lambda < 0$  of $H$ there holds
\begin{equation}\label{eq:ex1}
| \realpart \lambda | \leq C \| q \|_{n}^{2}
\end{equation}
 with a constant $C$  depending only on $n$, namely, $C = 4^{- 1} (1 - n^{- 1})^{n - 1}.$  However, in this case, $\alpha = n/(n + 1)$   and $\beta = \gamma^{'} = 2 n/(n + 1),$   hence the estimate \eqref{eq:ex1} also holds true with the constant  $C = n (1 - n)^{n - 1}/(n + 1)^{n + 1}$  provided that 
$$(A_{\alpha} A_{\beta} A_{\gamma^{'}})^{n} = \frac{2}{\sqrt{n}} \biggl( \frac{n}{n + 1} \biggl)^{(n + 1)/2},$$
as is easily checked. Obviously,  $A_{\alpha} A_{\beta} A_{\gamma^{'}} < 1.$
\end{example}

 \textbf{2.}   In the previous argument somewhat it was  involved the {\em heat kernel}  associated to  the Laplacian on $\reals^{n}.$  In fact  it could  be equivalently used the kernel 
\begin{equation}\label{eq:heatkernel}
h (x, y; t)  =  (4 \pi t)^{- n/2}  e^{- | x - y |^{2}/4 t}, \ t > 0,
\end{equation}
representing the (one-parameter) semi-group  $e^{- t H_{0}}   (0 \leq t < \infty).$  More exactly,   $e^{- t H_{0}}$  is represented by the integral operator with  the kernel \eqref{eq:heatkernel}, i.e.,
\begin{equation}\label{eq:intkernel}
(e^{- t H_{0}} u)  (x)  =  (4 \pi t)^{- n/2}  \int_{\reals^{n}}     e^{- | x - y |^{2}/4 t}   u (y) \ud y, \quad t > 0.
\end{equation}

The arguments similar to those used in proving  Theorem  \ref{thm:result2}  can be  applied  to obtain (under  suitable conditions) the estimate 
$$\| B  e^{- t H_{0}}  A u \|_{p}  \leq   (4 \pi t)^{- n/2  \alpha^{'}}   \alpha^{- n/2  \alpha}   \| a \|_{r}   \| b \|_{s} \| u \|_{p}, \quad u \in L_{p} (\reals^{n}).$$

 Then, from  the formula expressing the resolvent  $R (\lambda; H_{0})$ as the Laplace transform of the   semi-group  $e^{- t H_{0}}$  (see, for instance, \cite{hille-phillips}), i.e.,

\begin{equation}\label{eq:semigroup}
R (\lambda; H_{0}) =  \int_{0}^{\infty}   e^{\lambda t }    e^{- t H_{0}} \ud t, \quad \realpart  \lambda < 0,
\end{equation}
  we can further  estimate 
$$\| B R (\lambda; H_{0}) A  \| \leq   \int_{0}^{\infty}   e^{(\realpart \lambda) t}   \| B e^{- t H_{0}} A \| \ud t   $$
$$\leq  (4 \pi)^{- n/2  \alpha^{'}}   \alpha^{- n/2  \alpha}  \| a \|_{r}  \  \| a \|_{s}   \int_{0}^{\infty}  t^{- n/2 \alpha^{'}}  e^{(\realpart \lambda) t}  \ud t $$
$$\leq  (4 \pi)^{- n/2  \alpha^{'}}   \alpha^{- n/2  \alpha}   | \realpart \lambda |^{- 1 + n/2  \alpha^{'}} \ \Gamma (1 - n/2  \alpha^{'})   \| a \|_{r}    \| b \|_{s},$$
i.e.,
$$\| B R (\lambda; H_{0}) A  \| \leq  (4 \pi)^{- n/2  \alpha^{'}}   \alpha^{- n/2  \alpha}   | \realpart \lambda |^{- 1 + n/2  \alpha^{'}} \ \Gamma (1 - n/2  \alpha^{'})   \| a \|_{r}    \| b \|_{s},$$

and, thus, we come to the  same estimate as in \eqref{eq:res2}.

The next result concerns evaluation of  the  imaginary  part for a complex  eigenvalue $\lambda$ of $H.$  
 
\begin{thm}\label{thm:result3}
Let $n \geq 3,$  $1 < p < \infty,$  and let $q = a b,$  where 
$a \in L_{r} (\reals^{n}),   b \in  L_{s} (\reals^{n})$   for $r, s$ satisfying  
  $0 < r \leq  \infty,    \quad p  \leq s \leq  \infty,$   $ r^{- 1}  -  s^{- 1} = 1 - 2 p^{- 1},  \quad  2^{- 1} - p^{- 1} \leq  r^{- 1} \leq 1 - p^{- 1}$   and $r^{- 1} + s^{- 1} < 2 n^{- 1}.$  
Then, for any complex  eigenvalue $\lambda$   with   $\impart \lambda \neq 0$      of  the   Schr\"{o}dinger operator $H$,  considered acting in  the space $ L_{p} (\reals^{n})$,  there holds  
\begin{equation}\label{eq:res3}
 | \impart \lambda |^{\alpha}  \leq  (4 \pi)^{\alpha - 1}    \Gamma  ( \alpha)      \| a \|_{r}       \| b \|_{s},
\end{equation}
where $ \alpha  = 1  -   n (r^{- 1} + s^{- 1})/2.$
\end{thm}
\begin{proof}
The proof will depend upon a  modification of the argument used in proving the previous result. 
  Instead of \eqref{eq:semigroup} it will be used the formula expressing the resolvent $R (\lambda ; H_{0})$   as the Laplace transform of the  operator-group $e^{- i t H_{0}} \ (- \infty < t < \infty),$ namely
\begin{equation}\label{eq:opergroup}
R (\lambda; H_{0})  = i \int_{0}^{\infty}      e^{i  \lambda t} e^{- i t H_{0}} \ud t
\end{equation}
if,  for instance, $\impart \lambda > 0.$    First, we estimate the  norm $\| B  e^{- i t H_{0}}  A \|$  and then by using the formula  \eqref{eq:opergroup} we will derive estimation for  $ \impart \lambda $  (we  preserve notations made above).

As is known (cf., for instance, \cite{prosser} and also \cite{kato1}, Ch.IX),   for a fixed real $t, e^{- i t H_{0}} $ represents an integral  operator with the heat kernel (cf. \eqref{eq:heatkernel})

$$h (x, y; i t )  =   (4 \pi i t)^{- n/2}  e^{- | x - y |^{2}/4 i t}.$$

Writing
\begin{equation}\label{eq:hk}
(  e^{- i t H_{0}} A u)  (x)  = (4 \pi i t)^{- n/2}   e^{- | x |^{2}/4 i t}   \int_{\reals^{n}}   e^{- i \langle x,y \rangle/2 t}   e^{- | y |^{2}/4 i t}   a (y)    u  (y) \ud y,
\end{equation}
we argue as follows. 

We already know that 
$$\| A u \|_{\beta}  \leq \| a \|_{r}    \| u \|_{p},  \quad  \beta^{- 1}   =  r^{- 1} + p^{- 1}.$$
 It follows that for any  $u \in L_{p} (\reals^{n})$  the function $v$ defined by $v (y)= e^{- | y |^{2}/4 i t} a(y) u(y)$  belongs to $ L_{\beta } (\reals^{n}),$   and
\begin{equation}\label{eq:vnorm}
\| v \|_{\beta}  \leq \| a \|_{r}    \| u \|_{p}.
\end{equation}	
 
 Further, the integral on the right-hand side in \eqref{eq:hk} represents the function 
 $(2 \pi )^{n/2}   \hat{v}   (x/2 t),$   where $ \hat{v} $ denotes the Fourier transform of $v$.  
According to   the Hausdorf-Young theorem  (see, for instance, \cite{bergh-lofstrom}, Theorem 1.2.1) the  Fourier transform represents a bounded  operator from  $L_{\beta} (\reals^{n})$  to $L_{\beta^{'}} (\reals^{n})$  with  $1 \leq \beta \leq 2,$ and its norm is bounded by  $(2 \pi)^{- n/2 + n/\beta^{'}},$  i.e.,
\begin{equation}\label{eq:bdnorm}
\| \hat{v} \|_{\beta^{'}}  \leq  (2 \pi )^{- n/2 + n/\beta^{'}}   \| v \|_{\beta}.
\end{equation}
It follows that $\hat{v}  \in  L_{\beta^{'}} (\reals^{n})$ and, since
$$( e^{- i t H_{0}} A u) (x) = (4 \pi i t)^{- n/2}   e^{- | x |^{2}/4 i t}   (2 \pi )^{ n/2}  \hat{v} (x/2 t),$$
the function $ e^{- i t H_{0}} A u$ belongs to $ L_{\beta^{'}} (\reals^{n}).$  
Moreover, in view of \eqref{eq:vnorm} and \eqref{eq:bdnorm},
$$\| e^{- i t H_{0}} A u \|_{\beta^{'}}  =    (4 \pi t)^{- n/2}  (2 \pi)^{n/2}  \biggl( \int_{\reals^{n}}  |  \hat{v} (x/2 t)   |^{\beta^{'}}  \ud x \biggl)^{1/\beta^{'}}  $$

$$=   (4 \pi t)^{- n/2}  (2 \pi)^{n/2}  (2 t)^{n/\beta^{'}}  \| \hat{v} \|_{\beta^{'}}  \leq    (4 \pi t)^{- n/2}  (2 \pi)^{n/2}  (2 t)^{n/\beta^{'}}    (2 \pi )^{- n/2 + n/\beta^{'}}  \| {v} \|_{\beta}  $$

$$\leq   (4 \pi t)^{- n/2 + n/\beta^{'}}   \| a \|_{r} \ \| u \|_{p},$$
so that
$$\| e^{- i t H_{0}} A u \|_{\beta^{'}}  \leq (4 \pi t)^{- n/2 + n/\beta^{'}}     \| a \|_{r}   \| u \|_{p},  \ u \in  L_{p} (\reals^{n}).$$
On the other hand, since  $ r^{- 1}  -  s^{- 1}  = 1 - 2 p^{- 1},$ and since $\beta^{- 1} = r^{- 1} + p^{- 1},$
 one has
$ s^{- 1}  +   \beta^{' - 1}  = p^{- 1},$
that guarantees the boundedness of the operator of multiplication  $B$  regarded as an operator acting  from  $L_{ \beta^{'}} (\reals^{n})$  to $L_{p} (\reals^{n}).$   Moreover, 
$$\| B  v \|_{p}  \leq \| b \|_{s}  \| v \|_{\beta^{'}},  \quad v \in  L_{p^{'}} (\reals^{n}),$$

It is seen that for any $ u \in L_{p} (\reals^{n})$  the element $B  e^{- i t H_{0}} A u$ belongs to  $L_{p} (\reals^{n}),$  and

$$\| B  e^{- i t H_{0}}  A u \|_{p}  \leq   (4 \pi t)^{- n/2 + n/\beta^{'}}    \| a \|_{ r}    \| b \|_{ s}     \| u \|_{p}, \quad   u \in L_{p} (\reals^{n}).$$

Now, we apply \eqref{eq:opergroup} and for $\impart \lambda > 0$  we find 
$$\| B R (\lambda; H_{0}) A u \|_{p}   \leq  \int_{0}^{\infty}  e^{-(\impart \lambda) t}     \| B  e^{- i t H_{0}}  A u \|_{p}  \ud t $$

$$\leq (4 \pi )^{- n/2 + n/\beta^{'}}    \| a \|_{r}    \| b \|_{s}     \| u \|_{p}   \int_{0}^{\infty}  t^{- n/2 + n/\beta^{'}}  e^{-(\impart \lambda) t}  \ud t.$$

Next, we observe 
$1 - n/2 + n/\beta^{'}   = \alpha$
that was  assumed to be positive, and thus we can  apply the formula \eqref{eq:real}, due to of which, we have 

$$\int_{0}^{\infty}  t^{- n/2 + n/\beta^{'}}    e^{-(\impart \lambda) t}  \ud t  =    ( \impart \lambda )^{-\alpha} \Gamma  (\alpha).$$

Therefore, 

$$\| B R (\lambda; H_{0}) A  \|   \leq   (4 \pi)^{\alpha - 1}   ( \impart \lambda )^{- \alpha}   \Gamma (\alpha)   \| a \|_{r}  \| b \|_{s}.$$

For an eigenvalue $\lambda$ of $H$ it should be 
$$1 \leq (4 \pi)^{\alpha - 1}  ( \impart \lambda )^{- \alpha }  \Gamma (\alpha)  \| a \|_{r} \ \| b \|_{s},$$

that is \eqref{eq:res3}.

The estimate for the case $\impart \lambda < 0$  is treated similarly coming from the formula 

$$ R (\lambda; H_{0})  = - i  \int^{0}_{\infty}  e^{i \lambda t}    e^{- i H_{0} t}   \ud t, \quad \impart \lambda < 0.$$
\end{proof}

Notice that if $r = s$ in Theorem \ref{thm:result3}, it must be only $p = 2$ and $r > n.$  For this case  we have the following result.
\begin{cor}\label{cor:im}
Let $n \geq 3$, $r > n$,  and suppose  $ q \in L_{r/2} (\reals^{n}).$    Then any complex eigenvalue $\lambda$  with  $Im \lambda \neq 0$ of the Schr\"{o}dinger operator $H$  defined in the space $L_{2} (\reals^{n})$  satisfies 
\begin{equation}\label{eq:corim}
| \impart \lambda |^{1 - n/r}  \leq  (4 \pi)^{- n/r} \Gamma (1 - n/r)  \| q \|_{r/2}.
\end{equation}
\end{cor}

For the particular case when $r = 2 \gamma + n$ we have the following result (an analogous  result to that given by Corollary \ref{cor:res2}). 
\begin{cor}\label{cor:im1}
Let $n \geq 3, \gamma > 0$ and suppose  that
$ q \in L_{\gamma + n/2} (\reals^{n}).$
Then for any complex eigenvalue $\lambda$  with  $\impart \lambda \neq 0$ of the Schr\"{o}dinger operator   defined in  $L_{2} (\reals^{n})$  there holds 
\begin{equation}\label{eq:corim1}
| \impart \lambda |^{\gamma}  \leq  (4 \pi)^{- n/2} \Gamma \biggl( \frac{2 \gamma}{2 \gamma + n}  \biggl)^{\gamma + n/2}   \int_{\reals^{n}}  | q (x) |^{\gamma + n/2} \ud x.
\end{equation}
\end{cor}

\begin{rem}\label{rem:vdash}
 The estimate given by Theorem \ref{thm:result3} can be improved  upon a constant less than $1$.  The point  is that in proving Theorem \ref{thm:result3} it can be applied the sharp form of  the Hausdorff-Young theorem which is due  to K. I. Babenko \cite{babenko} (see also W. Beckner \cite{beckner} for the general case relevant for our purposes). According to Babenko's  result estimation \eqref{eq:bdnorm}, and hence \eqref{eq:res3} as well,  can be refined  upon a constant less than $1$, namely
$$\| \hat{v} \|_{\beta^{'}}  \leq  (2 \pi )^{- n/2 + n/\beta^{'}}  A \| v \|_{\beta},$$
where $A = (\beta^{1/\beta}/\beta^{' 1/\beta^{'}})^{n/2}.$   It is always $A \leq 1$ provided of $ 1 \leq \beta \leq 2$, and it is strictly less than $1$ if $\beta$ is chosen such that $1 < \beta < 2$. The same concerns  estimates \eqref{eq:corim}  and \eqref{eq:corim1}.
\end{rem} 

\textbf{ 3.}  The norm evaluation   for the operators $B R (\lambda; H_{0}) A$  for $\lambda \in \complex \setminus [0, \infty)$  can be carried out representing the resolvent of $H_{0}$  in terms of the Fourier transform.  Namely, it can  use the following equality 
\begin{equation}\label{eq:br}
B R (\lambda; H_{0}) A = B F^{-  1} \widehat{R (\lambda; H_{0})} F A,
\end{equation}
where it is denoted
$$ \widehat{R (\lambda; H_{0})} =  F R (\lambda; H_{0}) F^{-  1}$$
($F,  F^{-  1}$ denote the Fourier operators).  Clearly, $ \widehat{R (\lambda; H_{0})}$ represents the multiplication operator by $(| \xi |^{2} - \lambda)^{- 1},$ i.e., 
$$ \widehat{R (\lambda; H_{0})}  \hat{u} (\xi)  =   (| \xi |^{2} - \lambda)^{- 1}   \hat{u} (\xi), \quad \xi \in  \reals^{n}.$$
On $L_{2} (\reals^{n})$   the mentioned relations are  obviously  true.  However, we will use them for the spaces   $L_{p} (\reals^{n})$  with $p \neq 2$, as well, preserving the same notations as in the Hilbert space case $p = 2.$ 

As before, by assuming  that $a \in  L_{r} (\reals^{n})$  and  $b \in L_{s} (\reals^{n})$   $(0 < r, s \leq \infty),$  we choose $\beta > 0$  and $\gamma > 0$ such that 
\begin{equation}\label{eq:normAu}
\|  A u \|_{\beta}  \leq \| a \|_{r}  \ \| u \|_{p}, \quad  \beta^{- 1} =  r^{- 1} + p^{- 1},
\end{equation}
\begin{equation}\label{eq:normBv}
\|  B v \|_{p}  \leq \| b \|_{ s}  \ \| v \|_{\gamma}, \quad  \quad  p^{- 1} =  s^{- 1} + \gamma^{- 1}.
\end{equation}

According to the Hausdorff-Young theorem,  if $1 \leq \beta \leq 2$, the Fourier transform $F$ represents a bounded operator from  $L_{\beta} (\reals^{n})$  to $L_{\beta^{'}} (\reals^{n})$   the norm of which is bounded by $(2 \pi)^{-n/2 + n/\beta^{'}}$,  i.e.,
\begin{equation}\label{eq:normf}
\|  F f \|_{\beta^{'}} \leq  (2 \pi)^{-n/2 + n/\beta^{'}}    \| f \|_{\beta}.
\end{equation}
The same concerns the inverse Fourier transform $F^{- 1}$ considered as an operator acting from   $L_{\gamma^{'}} (\reals^{n})$  to  $L_{\gamma} (\reals^{n}).$  If $1 \leq  \gamma^{'}  \leq 2,$  that is  equivalent to $2 \leq \gamma \leq \infty$, we have 
\begin{equation}\label{eq:normfinv}
\|  F^{- 1} g \|_{\gamma} \leq  (2 \pi)^{-n/2 + n/\gamma}    \| g \|_{\gamma^{'}}.
\end{equation}

Now, we take $\alpha, 0 < \alpha \leq \infty,$  such that 
\begin{equation}\label{eq:alfb}
\gamma^{' - 1}  = \alpha^{- 1}  + \beta^{' - 1},
\end{equation}
equivalently,     $\alpha^{- 1} =  r^{- 1}  +    s^{- 1},$
and evaluate the $L_{\alpha}$-norm of the function $h (\cdot; \lambda)$
defined by
$$h (\xi; \lambda)  = ( | \xi |^{2}   - \lambda)^{- 1}, \quad \xi \in  \reals^{n}.$$

For $\alpha \neq \infty$  we  have 

$$\| h (\cdot; \lambda)  \|_{\alpha}^{\alpha}  =  \int_{ \reals^{n}}  | | \xi |^{2} - \lambda |^{- \alpha} \ud \xi =  \int_{0}^{\infty}  \int_{S^{n - 1}}  \rho^{n - 1}  | \rho^{2} - \lambda |^{- \alpha} \ud \rho \ud \omega ?$$
$$=  mes (S^{n - 1})     \int_{0}^{\infty}   \rho^{n - 1}  | \rho^{2} - \lambda |^{- \alpha} \ud \rho,$$

where $mes (S^{n - 1})  =  2 \pi^{n/2} / \Gamma (n/2)$  is the surface measure of the unit sphere $ S^{n - 1}$  in $\reals^{n}.$  
  Therefore, 
\begin{equation}\label{eq:normh}
\| h (\cdot; \lambda)  \|_{\alpha}^{\alpha}  = 2 \pi^{n/2} / \Gamma (n/2)  \int_{0}^{\infty}   \rho^{n - 1}  | \rho^{2} - \lambda |^{- \alpha} \ud \rho.
\end{equation}

If, we are particularly interesting in estimation  of negative eigenvalues,  we let that $\realpart \lambda < 0$ and evaluate the integral in \eqref{eq:normh} as follows.  First we observe that 

$$ | \rho^{2} - \lambda |^{- 1} \leq  (\rho^{2} - \realpart \lambda)^{- 1},$$
and then by setting $\rho^{2} = t$  we obtain 
$$\| h (\cdot; \lambda)  \|_{\alpha}^{\alpha}  \leq  \frac{\pi^{n/2} | \realpart \lambda |^{- \alpha}}{\Gamma (n/2)}  \int_{0}^{\infty}  \frac{ t^{n/2 - 1}}{( | \realpart \lambda |^{- 1} t + 1)^{\alpha}} \ud t.$$
By supposing $\alpha > n/2$  the formula (\cite{gradshteyn-ryzhik}, 3.194.3.)
$$ \int_{0}^{\infty}  \frac{x^{\mu - 1}}{(1 + \beta x)^{\nu}} \ud x = \beta^{- \mu} B (\mu, \nu - \mu), \quad \ | arg \beta | < \pi, \quad \realpart \nu > \realpart \mu > 0$$
($B (x, y)$  denotes the  beta function), 
can be applied.  We get

$$\| h (\cdot; \lambda)  \|_{\alpha}^{\alpha}  \leq  \pi^{n/2} (\Gamma (n/2) )^{- 1}  | \realpart \lambda |^{n/2 - \alpha}  B ( n/2, \alpha - n/2),$$
or,  in view of the functional relation  between  beta  and gamma functions,
\begin{equation}\label{eq:normh1}
\| h (\cdot; \lambda)  \|_{\alpha}^{\alpha}  \leq	  \pi^{n/2}    | \realpart \lambda |^{n/2 - \alpha} \  \Gamma (\alpha - n/2) /  \Gamma (\alpha). 
\end{equation}

Thus, for $\alpha > n/2$ the function $ h (\cdot; \lambda)$ belongs  to the space $L_{\alpha}  (\reals^{n})$ and, since \eqref{eq:alfb}, it follows that the operator of multiplication $\widehat{R  (\lambda; H_{0})}$  is bounded as an operator acting from   $L_{\beta^{'}}  (\reals^{n})$  to  $L_{\gamma^{'}}  (\reals^{n})$,   and, due to of \eqref{eq:normh1},  there holds 

\begin{equation}\label{eq:normR}
\| \widehat{R ( \lambda; H_{0})} f  \|_{\gamma^{'}}  \leq 	  \pi^{n/2 \alpha}    | \realpart \lambda |^{n/2 \alpha - 1} \  (\Gamma (\alpha - n/2) /  \Gamma (\alpha))^{1/\alpha} \| f \|_{\beta^{'}}. 
\end{equation}
  
In this way we obtain (cf. \eqref{eq:normAu} - \eqref{eq:normfinv}, \eqref{eq:normR})

$$\| B R (\lambda; H_{0}) A   \|   \leq (2 \pi)^{- n/\alpha}   \pi^{n/2 \alpha}    | \realpart \lambda |^{n/2 \alpha - 1} \  (\Gamma (\alpha - n/2) / \Gamma (\alpha))^{1/\alpha} \| a \|_{r}    \| b \|_{s}.$$

Therefore, for an eigenvalue $\lambda$ of $H$, it should by fulfilled
$$1 \leq   (2 \pi)^{- n/\alpha}   \pi^{n/2 \alpha}    | \realpart \lambda |^{n/2 \alpha - 1} /  (\Gamma (\alpha - n/2) / \Gamma (\alpha))^{1/\alpha} \| a \|_{r}    \| b \|_{s},$$

or, equivalently,
\begin{equation}\label{eq:relam}
| \realpart \lambda |^{1 - n/2 \alpha}  \leq  (4  \pi)^{- n/2 \alpha}  (\Gamma (\alpha - n/2)  /  \Gamma (\alpha))^{1/\alpha} \| a \|_{r}    \| b \|_{s}.
\end{equation}

In the extremal case $\alpha = \infty$, that is only happen  if $r = s = \infty$  (recall that  $\alpha^{- 1} =  r^{- 1} +  s^{- 1}$),  there  holds

$$\| h (\cdot; \lambda)  \|_{\infty}  = \sup_{\xi \in \reals^{n}}  | | \xi |^{2}  - \lambda |^{- 1} \leq  \sup_{\rho > 0} (\rho^{2}  -  \realpart \lambda)^{- 1} = |  \realpart \lambda |^{- 1},  $$
i.e.,
$$\| h (\cdot; \lambda)  \|_{\infty}  \leq  | \realpart \lambda |^{- 1}.$$
  In accordance with this evaluation, one follows 
\begin{equation}\label{eq:relam1}
 | \realpart \lambda |	 \leq  \| a \|_{\infty}   \| b \|_{\infty},
\end{equation}
a natural  estimate for eigenvalues occurred outside  of the continuous spectrum of $H_{0}$ by bounded  perturbations.

Note that the restriction $1 \leq \beta \leq 2$  is  equivalent to $2^{- 1} -  p^{- 1} \leq r^{- 1} \leq  1 - p^{- 1},$   while     $1 \leq \gamma^{'} \leq 2$  to   $2^{- 1} +  p^{- 1} \leq s^{- 1} \leq  p^{- 1},$  and  $\alpha > n/2$  to $r^{- 1} + s^{- 1} < 2 n^{- 1}.$

We have proved the following result.

\begin{thm}\label{thm:result4}
Let $1 < p < \infty $, and let $q =  a b,$  where 
$ a \in L_{r}  (\reals^{n}),$  \   $b \in  L_{s}  (\reals^{n})$  for $r, s$  satisfying 
    $0 < r \leq \infty,$ \ $0  < s \leq \infty,$  \quad  $2^{- 1} -  p^{- 1} \leq r^{- 1} \leq 1 - p^{- 1},$    $- 2^{- 1} +  p^{- 1} \leq s^{- 1} \leq  p^{- 1}$,  and  $r^{- 1} +  s^{- 1} < 2 n^{- 1}.$    Then, for any  eigenvalue $\lambda$  with $\realpart \lambda < 0$  of the Schr\"{o}dinger  operator $H,$ considered acting in  the space $ L_{p}  (\reals^{n}),$  there holds 

\begin{equation}\label{eq:result4}
 | \realpart \lambda |^{\alpha - n/2}   \leq C    (n, \alpha)  \| a \|_{r}^{\alpha} \   \| b \|_{s}^{\alpha},
\end{equation}
where  $C (n, \alpha)  =  ( 4  \pi)^{- n/2}  \Gamma (\alpha - n/2)  / \Gamma (\alpha)$,   $\alpha = ( r^{- 1} +  s^{-1})^{- 1}.$

For $r = s = \infty$   there holds \eqref{eq:relam1}.
\end{thm}

For the particular case $n = 1,  p = 2$  and $r = s = 2$  one  has $\alpha = 1$     and    $C =  1/2,$
hence, in view of \eqref{eq:result4}, the following  estimate
\begin{equation}\label{eq:realpart}
 | \realpart \lambda |^{1/2}   \leq  \frac{1}{2}   \| V \|_{1} \   \biggl( =  \frac{1}{2} \int_{- \infty}^{\infty} | V (x) | \ud x \biggl)
\end{equation}
holds true for any eigenvalue $\lambda$  of $H$ with $\realpart \lambda < 0.$

The obtained evaluation \eqref{eq:realpart} corresponds  to the well-known result of L. Spruch (mentioned in \cite{keller})   concerning negative eigenvalues of the one-dimensional self-adjoint  Schr\"{o}dinger operator considered in $L_{2}  (\reals)$.
  For other related results see \cite{abramov-aslanyan-davies}, \cite{davies-nath}, \cite{frank-laptev-seiringer}, \cite{frank-laptev-lieb-seiringer}, \cite{laptev-safronov} and \cite{safronov1}.

Theorem \ref{thm:result4} implies more general result (cf. also Corollary \ref{cor:res22}).	
\begin{cor}\label{cor:gen}
Let $\gamma > 0$ for $n \geq 2$ and $\gamma \geq 1/2$ for $n = 1.$  If   $q \in  L_{\gamma + n/2}  (\reals^{n}),$ then every  eigenvalue $\lambda$  with $\realpart \lambda < 0$  of the Schr\"{o}dinger  operator $H$  defined in   $ L_{2}  (\reals^{n})$   satisfies
\begin{equation}\label{eq:corgen}
 | \realpart \lambda |^{\gamma}   \leq     (4 \pi)^{- n/2}  \frac{\Gamma (\gamma)}{\Gamma (\gamma + n/2)}   \int_{\reals^{n}} | q (x) |^{\gamma + n/2}  \ud x.
\end{equation}
\end{cor}

A rigorous evaluation of the integral  on the right-hand side of \eqref{eq:normh} leads to more  exact estimates for the perturbed eigenvalues.  To this end, we let $\lambda = | \lambda | e^{i \theta} (0 < \theta < 2 \pi)$    and put $\rho^{2} = | \lambda  | t.$ Then 

$$ \int_{0}^{\infty}   \frac{\rho^{n - 1}}{|  \rho^{2} - \lambda |^{\alpha}} \ud \rho   =
  \frac{1}{2}  | \lambda |^{n/2 - \alpha}  \int_{0}^{\infty}   \frac{t^{n/2  - 1}}{(t^{2} - 2 t \cos \theta + 1)^{\alpha/2}}  \ud t.$$

If $n/2 < \alpha,$ it can be applied the formula (\cite{gradshteyn-ryzhik}; 3.252.10.)

$$\int_{0}^{\infty}   \frac{x^{\mu  - 1}}{(x^{2} + 2 x \cos t + 1)^{\nu}}  \ud x =  (2 \sin t)^{\nu - 1/2}      \Gamma (\nu + 1/2 )  B (\mu, 2 \nu - \mu)  P_{\mu - \nu - 1/2}^{1/2 - \nu}  (\cos t)$$
$$(- \pi < t < \pi, \quad 0 < \realpart \mu < \realpart \ 2 \nu),$$
where $ P_{\mu}^{\nu} (z)    (- 1 \leq z \leq 1)$  denote for the spherical  harmonics of the first kind  (\cite{gradshteyn-ryzhik}; 8.7 - 8.8).  As a result we have
\begin{equation}\label{eq:normhal}
\| h (\cdot; \lambda)  \|_{\alpha} =  \pi^{n/2 \alpha}  | \lambda |^{n/2 \alpha - 1}  I (n, \alpha, \theta),
\end{equation}
where
$$I (n, \alpha, \theta) =  (2 \sin \theta)^{1/2 - 1/2 \alpha}   \biggl(  \frac{\Gamma  (\alpha/2 + 1/2)  \Gamma  (\alpha - n/2)}{\Gamma  (\alpha)}   P_{n/2 - \alpha/2 - 1/2}^{1/2 - \alpha/2}   (- \cos \theta) \biggl)^{1/2},$$

and hence

$$\| B R (\lambda; H_{0}) A u \|_{p}  \leq   (4 \pi^{- n/2 \alpha})  | \lambda |^{n/2 \alpha - 1}  I (n, \alpha, \theta)  \| a \|_{r}    \| b \|_{s}  \| u \|_{p}$$
(note that   $(- n/2 + n/\beta^{'}) + (- n/2 + n/\gamma)  + n/2 \alpha =  -n/2 \alpha$).

Therefore, we obtain the following result.
 \begin{thm}\label{thm:result5}
Under the same assumptions as in Theorem \ref{thm:result4}  for any eigenvalue $\lambda \in \complex \setminus [0, \infty)$  of the Schr\"{o}dinger operator $H$, considering acting in the space  $ L_{p}  (\reals^{n}),$  there holds the estimation
	\begin{equation}\label{eq:result5}
 | \lambda |^{\alpha - n/2}   \leq C (n, \alpha, \theta)   \| a \|_{r}^{\alpha}    \| b \|_{s}^{\alpha}, 
\end{equation}
where 
$C (n,  \alpha, \theta) = (4 \pi)^{- n/2} I  (n,  \alpha, \theta)^{\alpha} $  and $I  (n,  \alpha, \theta)$   as in \eqref{eq:normhal}.
\end{thm}
\begin{rem}\label{rem:res5}
The estimate \eqref{eq:result4}  and, of course, \eqref{eq:result5} as well can be improved upon the constant $A_{\beta} A_{\gamma^{'}}   ( = (\beta^{1/\beta}  \gamma^{' 1/\gamma^{'}} / \beta^{' 1/\beta^{'}} \gamma^{1/\gamma})^{n/2})$  due to the sharp form of the Hausdorff-Young theorem \cite{babenko} (cf. Remark \ref{rem:vdash}).
\end{rem}

\section{Polyharmonic operators}\label{sec:generalization}

We will extend the estimates  established previously to the operators of  the form 
$$H = (- \Delta)^{m} + q $$
in which (the potential) $q$ is a complex-valued function, and $m$  is an arbitrary positive  real number.
Unperturbed operator
$$H_{0}  =  (- \Delta)^{m}$$
can be comprehend, as
$$(H_{0} u)   (x) =  \int_{\reals^{n}}  | \xi |^{2 m}  \hat{u} (\xi) e^{- i \langle x, \xi \rangle} \ud \xi $$
defined,  for instance,   in $L_{2} (\reals^{n})$  on its  maximal domain consisting of all functions $u \in  L_{2} (\reals^{n})$  such that $H_{0} u \in  L_{2} (\reals^{n})$   (or, what is the same, $\hat{v}$  determined by  $\hat{v} (\xi) =  | \xi |^{2 m}  \hat{u} (\xi) $ belongs to $L_{2} (\reals^{n})$);  $\hat{u}$ denotes  the Fourier transform of $u$.  $H_{0}$ can be treated  upon a unitary equivalence (by the Fourier  transform) as the operator of multiplication by $ | \xi |^{2 m}$.  

In the space $ L_{p}  (\reals^{n})  \ ( 1 < p < \infty)$	   the operator $H$ can   be viewed as an elliptic operator  of  order  $2 m$ defined  on its domain the Sobolev space  $ W_{p}^{2 m}  (\reals^{n}).$  As  in preceding sections we assume that the potential $q$ admits a factorization $q = ab$  with $a, b$ for which conditions \eqref{eq:Ma}, \eqref{eq:Mb},  but with $0 < \nu < p^{'}  \kappa$      and $0 < \mu < p (m - \kappa)$ for some   $0 < \kappa < m,$ and \eqref{eq:qdecays} are satisfied.  Under  these conditions the operator  $(- \Delta)^{m} + q$  admits a closed extension $H$, let us denote it by  $H_{m, q}$,   to which the approach   for the evaluation of perturbed eigenvalues  proposed in Section \ref{sec:preliminaries} is applied.

Thus, in order to obtain estimation for the norm of $B R (\lambda; H_{0}) A$  (the operators $A, B$ are  defined as in previous subsections), we can use the relation \eqref{eq:br}, where
$$\widehat{R (\lambda; H_{0})}  \hat{u} (\xi) = (  | \xi |^{2 m} - \lambda)^{- 1}    \hat{u} (\xi), \quad \xi \in  \reals^{n}.$$

The arguments used in proving Theorems \ref{thm:result4} and \ref{thm:result5} can be applied, and as is seen we have only  to evaluate,  for appropriate $\alpha > 0,$ the $L_{\alpha}$-norm  of the function $h_{m} (\cdot; \lambda)$   defined by 
$$h_{m} (\xi; \lambda)  =    (  | \xi |^{2 m} - \lambda)^{- 1}, \quad \xi \in  \reals^{n}.$$

For any $\alpha, 0 < \alpha < \infty,$  we have  

$$\| h_{m} (\cdot; \lambda) \|_{\alpha}^{\alpha}  =  \int_{\reals^{n}}  \frac{d \xi}{| | \xi |^{2 m} - \lambda|^{\alpha}}   =       \frac{2 \pi^{n/2}}{\Gamma (n/2)}    \int_{0}^{\infty}    \frac{\rho^{n - 1}}{| \rho^{2 m} - \lambda |^{\alpha}}  \ud \rho.$$

  Writing $\lambda = | \lambda | e^{i \theta}  (0 < \theta < 2 \pi)$  and making the substitution $\rho^{2 m} = | \lambda | t,$ we obtain 
	
$$\int_{0}^{\infty}    \frac{\rho^{n - 1}}{| \rho^{2 m} - \lambda |^{\alpha}}  \ud \rho  =  
  \frac{1}{2 m}  | \lambda |^{n/2m - \alpha}   \int_{0}^{\infty}    \frac{t^{n/2m - 1}}{(t^{2} - 2 t \cos \theta + 1)^{\alpha/2}} \ud t.$$

  Assuming $n/2m < \alpha$  we apply again the formula (\cite{gradshteyn-ryzhik}; 3.252.10.), and obtain
	
Hence,
\begin{equation}\label{eq:hm}
\| h_{m} (\cdot; \lambda) \|_{\alpha}^{\alpha}  =  \frac{2 \pi^{n/2}}{\Gamma (n/2)}  \cdot \frac{1}{2 m}  | \lambda |^{n/2 m - \alpha}  I_{m} (n, \alpha, \theta),
\end{equation}
where 
$$ I_{m} (n, \alpha, \theta) =    (2 \sin  \theta)^{\alpha/2 - 1/2}  \Gamma (\alpha/2 + 1/2)  B (n/2 m, \alpha - n/2m)     P_{n/2m - \alpha/2 - 1/2}^{1/2 - \alpha/2}  (- \cos \theta).$$

Collecting all  evaluations we obtain the following result.

\begin{thm}\label{thm:result6}
Let  $1 < p < \infty,   m > 0$,   and let  $q = a b,$  where $a \in L_{r} (\reals^{n} )$,  $b \in  L_{s} (\reals^{n})$ for $r, s$ satisfying $0 < r \leq \infty,$ $0 <  s \leq \infty$, $2^{-1} -  p^{- 1} \leq r^{- 1} \leq 1 - p^{- 1},$ \quad $- 2^{- 1}  +  p^{- 1} \leq s^{- 1}  \leq  p^{- 1},$    and  $ r^{- 1} + s^{- 1} < 2 m n.$ 
Then, for  any eigenvalue $\lambda \in  \complex \setminus [0, \infty)$   of the  operator $H_{m, q}$, considered acting in   $L_{p} (\reals^{n}),$ there holds
\begin{equation}\label{eq:result6}
| \lambda |^{\alpha - n/2 m}  \leq C (n, m, \alpha, \theta)  \| a \|_{r}^{\alpha}    \| b \|_{s}^{\alpha},
\end{equation}
where    $C (n, m, \alpha, \theta)  =  (4 \pi)^{- n/2}  (m \Gamma (n/2))^{- 1}   I_{m}  (n, \alpha, \theta),  \quad  I_{m}  (n, \alpha, \theta)$   is determined as  in \eqref{eq:hm}, and $\alpha = (r^{- 1} + s^{- 1})^{- 1}.$
\end{thm}

As a consequence of Theorem \ref{thm:result6}  we have  a result analogous to that given by Corollary \ref{cor:gen}.
\begin{cor}\label{cor:res6}
Let $\gamma > 0$ for $n \geq 2 m$ and $\gamma \geq 1 - n/2 m$ for $n < 2 m.$  If $q \in L_{\gamma + n/2 m} (\reals^{n} ),$ then every eigenvalue $\lambda \in \complex \setminus [0, \infty)$  of the operator $H_{m, q}$    defined in $ L_{2} (\reals^{n} )$ satisfies
\begin{equation}\label{eq:corres6}
| \lambda |^{\gamma}  \leq C (n, m, \alpha, \theta) \int_{\reals^{n}}  | q (x) |^{\gamma + n/2 m}   \ud x,
\end{equation}
where  $C (n, m, \alpha, \theta)$ is as in (2.54).
\end{cor}
\begin{rem}\label{rem:res6}
Similarly, as for estimates \eqref{eq:result4}  and \eqref{eq:result5},  the estimate \eqref{eq:result6}  and hence \eqref{eq:corres6} can be improved upon  the constant $A_{\beta} A_{\gamma^{'}}$   (see Remark \ref{rem:res5}).
\end{rem}


\section*{Acknowledgments}

The author wishes to express her gratitudes to Professor Ari Laptev for formulating the problem and for many useful discussions.

\bibliography{complex}
\bibliographystyle{alpha}

\end{document}